\documentclass[12pt]{amsart}
\usepackage{latexsym,amssymb,amsmath,amsthm,amscd,enumerate,graphicx,
longtable}

\makeatletter
\@namedef{subjclassname@2010}{%
  \textup{2010} Mathematics Subject Classification}
\makeatother

\newtheorem{thm}{Theorem}[section]
\newtheorem{cor}[thm]{Corollary}
\newtheorem{lem}[thm]{Lemma}
\newtheorem{prop}[thm]{Proposition}

\theoremstyle{definition}

\numberwithin{equation}{section}

\newcommand{\cD}{{\mathcal D}}

\newcommand{\cF}{{\mathcal F}}

\newcommand{\cR}{{\mathcal R}}

\newcommand{\cdq}{{\mathcal D\!Q}}
\newcommand{\cdr}{{\mathcal D\!R}}

\newcommand{\R}{{\mathbb R}}

\newcommand{\Z}{{\mathbb Z}}

\def\al{\alpha}
\def\bt{\beta}

\def\sg{\sigma}

\def\om{\omega}

\def\0{\emptyset}
\def\1{\textbf{\rm 1}}
\def\6{\partial}
\def\8{\infty}

\def\lt{\left}
\def\rt{\right}

\def\ol{\overline}

\def\ch{\operatorname{ch}}

\begin{document}

\title{The $n$-linear embedding theorem for dyadic rectangles}

\author[H.~Tanaka]{Hitoshi~Tanaka}
\address{
Research and Support Center on Higher Education for the hearing and Visually Impaired, 
National University Corporation Tsukuba University of Technology, 
Kasuga 4-12-7, 
Tsukuba City 305-8521, 
Japan
}
\email{htanaka@k.tsukuba-tech.ac.jp}
\author[K.~Yabuta]{K\^{o}z\^{o}~Yabuta}
\address{
Research center for Mathematical Sciences,
Kwansei Gakuin University,
Gakuen 2-1, 
Sanda 669-1337, 
Japan
}
\email{kyabuta3@kwansei.ac.jp}
\thanks{
The first author is supported by 
Grant-in-Aid for Scientific Research (C) (15K04918), 
the Japan Society for the Promotion of Science. 
The second author was supported partly by 
Grant-in-Aid for Scientific Research (C) Nr.~23540228, 
the Japan Society for the Promotion of Science. 
}

\date{}

\begin{abstract}
Let $\sg_i$, $i=1,\ldots,n$, denote 
reverse doubling weights on $\R^d$, 
let $\cdr(\R^d)$ denote the set of all dyadic rectangles on $\R^d$ 
(Cartesian products of usual dyadic intervals) 
and let $K:\,\cdr(\R^d)\to[0,\8)$ be a~map. 
In this paper we give the $n$-linear embedding theorem for dyadic rectangles. 
That is, we prove 
the $n$-linear embedding inequality for dyadic rectangles
\[
\sum_{R\in\cdr(\R^d)}
K(R)\prod_{i=1}^n\lt|\int_{R}f_i\,{\rm d}\sg_i\rt|
\le C
\prod_{i=1}^n
\|f_i\|_{L^{p_i}(\sg_i)}
\]
can be characterized by simple testing condition 
\[
K(R)\prod_{i=1}^n\sg_i(R)
\le C
\prod_{i=1}^n\sg_i(R)^{\frac{1}{p_i}}
\quad R\in\cdr(\R^d),
\]
in the range $1<p_i<\8$ and 
$\sum_{i=1}^n\frac{1}{p_i}>1$.
As a~corollary to this theorem, 
for reverse doubling weights, 
we verify a~necessary and sufficient condition 
for which the weighted norm inequality for the multilinear strong positive dyadic operator and for strong fractional integral operator to hold.
\end{abstract}

\maketitle

\section{Introduction}\label{sec1}
The purpose of this paper is to prove the $n$-linear embedding theorem for dyadic rectangles. 
We will denote by $\cdq(\R^d)$ the family of all dyadic cubes 
$Q=2^{-k}(m+[0,1)^d)$, 
$k\in\Z,\,m\in\Z^d$. 
We will denote by $\cdr(\R^d)$ the family of all dyadic rectangles on $\R^d$. 
The dyadic rectangle we mean that the Cartesian product of the dyadic intervals $\cdq(\R)$. 
Throughout this paper 
$n$ stands for an~integer which is greater than one. 

Through a series of works 
\cite{HHL,Hy,LSU,NTV,Ta1,Ta2,Ta3,Tr},
one perfectly characterizes 
the $n$-linear embedding inequality for dyadic cubes. 
Let $\sg_i$, $i=1,\ldots,n$, denote 
positive Borel measures on $\R^d$ and 
let $K:\,\cdq(\R^d)\to[0,\8)$ be a~map. 
The $n$-linear embedding inequality for dyadic cubes 
\begin{equation}\label{1.1}
\sum_{Q\in\cdq(\R^d)}
K(Q)\prod_{i=1}^n\lt|\int_{Q}f_i\,{\rm d}\sg_i\rt|
\le C
\prod_{i=1}^n
\|f_i\|_{L^{p_i}({\rm d}\sg_i)}
\end{equation}
can be characterized in the full range 
$1<p_i<\8$. 
The $n$-linear embedding theorem \eqref{1.1}, 
either can be reduced to the (localized) $(n-1)$-linear embedding theorems, 
or characterized by certain $n$-weight discrete Wolff potential conditions. 
The division line is whether 
the exponents $p_1,\ldots,p_n$ 
are in the super-dual range 
$\sum_{i=1}^n\frac{1}{p_i}\ge 1$ 
or in the strictly sub-dual range  
$\sum_{i=1}^n\frac{1}{p_i}<1$. 
The inner workings of each range seem to be rather different 
(see \cite{Ta3}). 
The main technique used to is that of 
\lq\lq parallel corona" decomposition 
from the work of Lacey et al. \cite{LSSU} 
on the two-weight boundedness of the Hilbert transform. 
However, this powerful technique deeply depends on the structure of dyadic cubes 
and can not apply dyadic rectangles. 
It is natural to consider what happens in the case $\cdr(\R^d)$ 
and the partial answer is given in this paper. 

By weights we will always mean 
nonnegative, 
locally integrable functions 
which are positive on a set of positive measure. 
Given a~measurable set $E$ and a~weight $\om$, 
we will use 
$\om(E)$ to denote $\int_{E}\om\,{\rm d}x$. 
By $1_{E}$ we stand for the characteristic function of $E$. 

Let $1\le p<\8$ and $\om$ be a~weight. 
We define the weighted Lebesgue space $L^p(\om)$ to be a~Banach space equipped with the norm 
\[
\|f\|_{L^p(\om)}
=
\lt(\int_{\R^d}|f|^p\,{\rm d}\om\rt)^{\frac{1}{p}},
\]
where we have used ${\rm d}\om:=\om\,{\rm d}x$. 
Given $1<p<\8$, $p'=\frac{p}{p-1}$ 
will denote the conjugate exponent of $p$. 

Let $\cR(\R^d)$ denote the set of all rectangles in $\R^d$ with sides parallel to the coordinate axes. 
We say that a~weight $\om$ is \lq\lq reverse doubling weight" 
if it satisfies that 
there is a~constant $\bt>1$ such that 
$\bt\om(R')\le\om(R)$ 
for any $R',R\in\cR(\R^d)$ where 
$R'$ is the two equal division of $R$. 
We shall prove the following theorem. 

\begin{thm}\label{thm1.1}
Let $1<p_i<\8$ and 
$\sum_{i=1}^n\frac{1}{p_i}>1$.
Let $K:\,\cdr(\R^d)\to[0,\8)$ be a~map 
and let $\sg_i$, $i=1,\ldots,n$, be reverse doubling weights on $\R^d$. 
The following statements are equivalent\text{:}

\begin{itemize}
\item[{\rm(a)}] 
The $n$-linear embedding inequality for dyadic rectangles 
\begin{equation}\label{1.2}
\sum_{R\in\cdr(\R^d)}
K(R)\prod_{i=1}^n\lt|\int_{R}f_i\,{\rm d}\sg_i\rt|
\le c_1
\prod_{i=1}^n
\|f_i\|_{L^{p_i}(\sg_i)}
\end{equation}
holds for all 
$f_i\in L^{p_i}(\sg_i)$, 
$i=1,\ldots,n$;
\item[{\rm(b)}] 
The testing condition 
\begin{equation}\label{1.3}
K(R)\prod_{i=1}^n\sg_i(R)
\le c_2
\prod_{i=1}^n\sg_i(R)^{\frac{1}{p_i}}
\end{equation}
holds for all dyadic rectangles $R\in\cdr(\R^d)$.
\end{itemize}

\noindent
Moreover, 
the least possible constants $c_1$ and $c_2$ are equivalent.
\end{thm}

\begin{cor}\label{cor1.2}
Let $1<p_i<\8$, $1<q<\8$ and 
$\sum_{i=1}^n\frac{1}{p_i}>\frac1q$. 
Let $K:\,\cdr(\R^d)\to[0,\8)$ be a~map and 
let $\sg_i$, $i=1,\ldots,n$, 
and $\om$ be reverse doubling weights on $\R^d$. 
The following statements are equivalent\text{:}

\begin{itemize}
\item[{\rm(a)}] 
The weighted norm inequality for multilinear strong positive operator 
\begin{equation}\label{1.4}
\|T_{K}(f_1,\ldots,f_n)\|_{L^q(\om)}
\le c_1
\prod_{i=1}^n
\|f_i\|_{L^{p_i}(\sg_i^{1-p_i})}
\end{equation}
holds for all 
$f_i\in L^{p_i}(\sg_i^{1-p_i})$, 
$i=1,\ldots,n$;
Here, 
\[
T_{K}(f_1,\ldots,f_n)
:=
\sum_{R\in\cdr(\R^d)}
K(R)1_{R}
\prod_{i=1}^n\int_{R}f_i\,{\rm d}x.
\]
\item[{\rm(b)}] 
The testing condition 
\begin{equation}\label{1.5}
K(R)\om(R)^{\frac1q}
\prod_{i=1}^n\sg_i(R)
\le c_2
\prod_{i=1}^n\sg_i(R)^{\frac{1}{p_i}}
\end{equation}
holds for all dyadic rectangles $R\in\cdr(\R^d)$.
\end{itemize}

\noindent
Moreover,
the least possible constants $c_1$ and $c_2$ are equivalent.
\end{cor}

In the last section 
we apply Corollary \ref{cor1.2} to strong fractional integral operator. 
Two-weight estimates 
for multilinear fractional strong maximal operator 
and for strong fractional integral operator 
see \cite{CXY,KM,SW}.

The letter $C$ will be used for constants
that may change from one occurrence to another.
Constants with subscripts, such as $C_1$, $C_2$, do not change
in different occurrences.
By $A\approx B$ we mean that 
$c^{-1}B\le A\le cB$ 
with some positive finite constant $c$ independent of appropriate quantities. 

\section{Lemmas}\label{sec2}
We need two lemmas and 
we will give their proofs for the sake of completeness. 

\begin{lem}\label{lem2.1}
Given a~weight $\sg$ in $\R^d$ and 
$1<p<q<\8$, 
the following statements are equivalent\text{:}

\begin{itemize}
\item[{\rm(a)}] 
The Carleson type embedding inequality for dyadic cubes 
\begin{equation}\label{2.1}
\sum_{Q\in\cdq(\R^d)}
\sg(Q)^{\frac{q}{p}}
\lt(\frac{1}{\sg(Q)}\int_{Q}f\,{\rm d}\sg\rt)^q
\le c_1
\lt(\int_{\R^d}f^p\,{\rm d}\sg\rt)^{\frac{q}{p}}
\end{equation}
holds for all nonnegative function 
$f\in L^p(\sg)$;
\item[{\rm(b)}] 
The testing condition 
\begin{equation}\label{2.2}
\sum_{
\substack{Q'\in\cdq(\R^d) \\ Q'\subset Q}
}
\sg(Q')^{\frac{q}{p}}
\le c_2
\sg(Q)^{\frac{q}{p}}
\end{equation}
holds for all cubes $Q\in\cdq(\R^d)$.
\end{itemize}

\noindent
Moreover, 
the least possible constants $c_1$ and $c_2$ are equivalent.
\end{lem}

\begin{proof}
The necessity \eqref{2.2} follows at once 
if we substitute the test function 
$f=1_{Q}$ into inequality \eqref{2.1}. 
To show that inequality \eqref{2.2} is sufficient, 
we fix a (big enough) dyadic cube $Q_0\in\cD(\R^d)$ 
and we prove the inequality 
\begin{equation}\label{2.3}
\sum_{
\substack{Q\in\cdq(\R^d) \\ Q\subset Q_0}
}
\sg(Q)^{\frac{q}{p}}
\lt(\frac{1}{\sg(Q)}\int_{Q}f\,{\rm d}\sg\rt)^q
\le C c_2
\lt(\int_{Q_0}f^p\,{\rm d}\sg\rt)^{\frac{q}{p}}.
\end{equation}

We define the collection of principal cubes $\cF$ 
for the pair $(f,\sg)$. Namely, 
\[
\cF:=\bigcup_{k=0}^{\8}\cF_k,
\]
where 
$\cF_0:=\{Q_0\}$,
\[
\cF_{k+1}
:=
\bigcup_{F\in\cF_k}\ch_{\cF}(F)
\]
and $\ch_{\cF}(F)$ is defined by 
the set of all maximal dyadic cubes $Q\subset F$ such that 
\[
\frac{1}{\sg(Q)}\int_{Q}f\,{\rm d}\sg
>
\frac{2}{\sg(F)}\int_{F}f\,{\rm d}\sg.
\]
Observe that
\begin{align*}
&
\sum_{F'\in\ch_{\cF}(F)}\sg(F')
\\ &\le
\lt(\frac{2}{\sg(F)}\int_{F}f\,{\rm d}\sg\rt)^{-1}
\sum_{F'\in \ch_{\cF}(F)}
\int_{F'}f\,{\rm d}\sg
\\ &\le
\lt(\frac{2}{\sg(F)}\int_{F}f\,{\rm d}\sg\rt)^{-1}
\int_{F}f\,{\rm d}\sg
=
\frac{\sg(F)}{2},
\end{align*}
and, hence, 
\begin{equation}\label{2.4}
\sg(E_{\cF}(F))
:=
\sg\lt(F\setminus\bigcup_{F'\in \ch_{\cF}(F)}F'\rt)
\ge
\frac{\sg(F)}{2},
\end{equation}
where the sets in the collection 
$\{E_{\cF}(F):\,F\in\cF\}$ 
are pairwise disjoint. 

We further define the stopping parent, 
for $Q\in\cdq(\R^d)$, 
\[
\pi_{\cF}(Q)
:=
\min\{F\supset Q:\,F\in\cF\}.
\]

Then we can rewrite the series in \eqref{2.3} as follows:
\begin{align*}
&
\sum_{Q\subset Q_0}
\sg(Q)^{\frac{q}{p}}
\lt(\frac{1}{\sg(Q)}\int_{Q}f\,{\rm d}\sg\rt)^q
\\ &=
\sum_{F\in\cF}
\sum_{Q:\,\pi_{\cF}(Q)=F}
\sg(Q)^{\frac{q}{p}}
\lt(\frac{1}{\sg(Q)}\int_{Q}f\,{\rm d}\sg\rt)^q
\\ &\le 
\sum_{F\in\cF}
\lt(\frac{2}{\sg(F)}\int_{F}f\,{\rm d}\sg\rt)^q
\sum_{Q:\,\pi_{\cF}(Q)=F}
\sg(Q)^{\frac{q}{p}}
\\ &\le 2^q c_2
\sum_{F\in\cF}
\lt(\frac{1}{\sg(F)}\int_{F}f\,{\rm d}\sg\rt)^q
\sg(F)^{\frac{q}{p}},
\end{align*}
where we have used the condition \eqref{2.2}.

Using 
$\|\cdot\|_{l^p}\ge\|\cdot\|_{l^q}$, 
for $0<p\le q<\8$, and 
\eqref{2.4} we can proceed further that 
\begin{align*}
&\le C c_2
\lt\{
\sum_{F\in\cF}
\lt(\frac{1}{\sg(F)}\int_{F}f\,{\rm d}\sg\rt)^p
\sg(F)
\rt\}^{\frac{q}{p}}
\\ &\le C c_2
\lt\{
\sum_{F\in\cF}
\lt(\frac{1}{\sg(F)}\int_{F}f\,{\rm d}\sg\rt)^p
\sg(E_{\cF}(F))
\rt\}^{\frac{q}{p}}
\\ &\le C c_2
\lt(\int_{Q_0}M_{\cdq}^{\sg}[f1_{Q_0}]^p\,{\rm d}\sg\rt)^{\frac{q}{p}}
\\ &\le C c_2
\lt(\int_{Q_0}f^p\,{\rm d}\sg\rt)^{\frac{q}{p}},
\end{align*}
where 
$M_{\cdq}^{\sg}$ stands for 
the dyadic Hardy-Littlewood maximal operator 
with respect to the measure ${\rm d}\sg$ and 
we have used its boundedness. 
This completes the proof. 
\end{proof}

We denote by $P_i$, $i=1,\ldots,d$, 
the projection on the $x_j$-axis. 
For 
the dyadic rectangle $R\in\cdr(\R^d)$, 
the dyadic interval $I\in\cdq(\R)$ 
and $j=1,\ldots,d$,
we define the dyadic rectangle 
\[
[R;\,I,j]
:=
\lt(\prod_{i=1}^{j-1}P_i(R)\rt)
\times I\times
\lt(\prod_{i=j+1}^dP_i(R)\rt).
\]

\begin{lem}\label{lem2.2}
Given a~weight $\sg$ in $\R^d$ and 
$1<p<q<\8$, 
the following statements are equivalent\text{:}

\begin{itemize}
\item[{\rm(a)}] 
The Carleson type embedding inequality for rectangles 
\begin{equation}\label{2.5}
\sum_{R\in\cdr(\R^d)}
\sg(R)^{\frac{q}{p}}
\lt(\frac{1}{\sg(R)}\int_{R}f\,{\rm d}\sg\rt)^q
\le c_1
\lt(\int_{\R^d}f^p\,{\rm d}\sg\rt)^{\frac{q}{p}}
\end{equation}
holds for all nonnegative function 
$f\in L^p(\sg)$;
\item[{\rm(b)}] 
The testing condition 
\begin{equation}\label{2.6}
\sum_{
\substack{I\in\cdq(\R) \\ I\subset P_j(R)}
}
\sg([R;\,I,j])^{\frac{q}{p}}
\le c_2
\sg(R)^{\frac{q}{p}}
\end{equation}
holds for all dyadic rectangles $R\in\cdr(\R^d)$ 
and $j=1,\ldots,d$.
\end{itemize}

\noindent
Moreover, 
the least possible constants $c_1$ and $c_2$ 
enjoy $c_1\le C c_2^d$ 
and $c_2\le c_1$.
\end{lem}

\begin{proof}
The necessity is clear, so we shall prove the sufficiency. 
We use induction on the dimension $d$. 
To do this, 
we assume that the lemma is true for the case $d-1$. 

We assume the weight $\sg$ in $\R^d$ 
satisfies the testing condition \eqref{2.6} 
($d$-dimensional case).
We will write 
$x=(x_1,\ldots,x_{d-1},x_d)=(\ol{x},x_d)$.

We need two observations. 
First, we verify that, 
for any dyadic interval $I_d\in\cdq(\R)$, 
if we let 
\[
v_{I_d}(\ol{x})
:=
\int_{I_d}\sg(\ol{x},x_d)\,{\rm d}x_d,
\]
then $v_{I_d}(\ol{x})$ satisfies 
the testing condition \eqref{2.6} 
($(d-1)$-dimensional case).
Indeed, 
for any $\ol{R}\in\cdr(\R^{d-1})$, 
setting $R=\ol{R}\times I_d$, 
we have that, for $j=1,\ldots,d-1$, 
\begin{align*}
&
\sum_{
\substack{I\in\cdq(\R) \\ I\subset P_j(\ol{R})}
}
v_{I_d}([\ol{R};\,I,j])^{\frac{q}{p}}
&=
\sum_{
\substack{I\in\cdq(\R) \\ I\subset P_j(R)}
}
\sg([R;\,I,j])^{\frac{q}{p}}
\\ &\le c_2
\sg(R)^{\frac{q}{p}}
=c_2
v_{I_d}(\ol{R})^{\frac{q}{p}}.
\end{align*}

We next verify that, 
for a.e. $\ol{x}\in\R^{d-1}$, 
if we let 
\[
v_{\ol{x}}(x_d)=\sg(\ol{x},x_d),
\]
then $v_{\ol{x}}(x_d)$ satisfies 
the testing condition \eqref{2.2} 
(one-dimensional case). 
We must prove that the inequality 
\begin{equation}\label{2.7}
\sum_{
\substack{I\in\cdq(\R) \\ I\subset I_d}
}
v_{\ol{x}}(I)^{\frac{q}{p}}
\le c_2
v_{\ol{x}}(I_d)^{\frac{q}{p}}
\end{equation}
holds for any $I_d\in\cdq(\R)$. 
For a~cube $\ol{Q}\in\cdq(\R^{d-1})$, 
it follows by setting 
$R=\ol{Q}\times I_d$ that 
\[
\sum_{
\substack{I\in\cdq(\R) \\ I\subset P_d(R)}
}
\sg([R;\,I,d])^{\frac{q}{p}}
\le c_2
\sg(R)^{\frac{q}{p}}.
\]
Dividing the both sides by the volume $|\ol{Q}|^{\frac{q}{p}}$, 
\[
\sum_{
\substack{I\in\cdq(\R) \\ I\subset P_d(R)}
}
\lt(
\frac{1}{|\ol{Q}|}
\int_{\ol{Q}\times I}
\sg(\ol{x},x_d)\,{\rm d}x_d\,{\rm d}\ol{x}
\rt)^{\frac{q}{p}}
\le c_2
\lt(
\frac{1}{|\ol{Q}|}
\int_{\ol{Q}\times I_d}
\sg(\ol{x},x_d)\,{\rm d}x_d\,{\rm d}\ol{x}
\rt)^{\frac{q}{p}}.
\]
In the both sides of this inequality, 
considering the Lebesgue point $\ol{y}$ 
with respect to the integral avarages over $\ol{Q}$, 
which exists a.e. in $\R^{d-1}$ 
because our argument is countable, 
and shrinking $\ol{Q}$ to $\ol{y}$, 
we obtain 
\[
\sum_{
\substack{I\in\cdq(\R) \\ I\subset I_d}
}
\lt(
\int_{I}\sg(\ol{y},x_d)\,{\rm d}x_d
\rt)^{\frac{q}{p}}
\le c_2
\lt(
\int_{I_d}\sg(\ol{y},x_d)\,{\rm d}x_d
\rt)^{\frac{q}{p}},
\]
which means \eqref{2.7}. 

By the use of these two observations 
we can prove the lemma. 

Fix a~nonnegative function $f\in L^p(\sg)$. 
We shall evaluate 
\[
\text{(i)}
:=
\sum_{I_d\in\cdq(\R)}
\sum_{\ol{R}\in\cdr(\R^{d-1})}
\sg(R)^{\frac{q}{p}}
\lt(
\frac{1}{\sg(R)}
\int_{R}f\,{\rm d}\sg
\rt)^q,
\]
where we have used $R=\ol{R}\times I_d$. 

There holds 
\begin{align*}
\text{(i)}
&=
\sum_{I_d\in\cdq(\R)}
\sum_{\ol{R}\in\cdr(\R^{d-1})}
v_{I_d}(\ol{R})^{\frac{q}{p}}
\\ &\quad\times
\lt(
\frac{1}{v_{I_d}(\ol{R})}
\int_{\ol{R}}
\lt(
\int_{I_d}
f(\ol{x},x_d)\sg(\ol{x},x_d)
\,{\rm d}x_d
v_{I_d}(\ol{x})^{-1}
\rt)
v_{I_d}(\ol{x})
\,{\rm d}\ol{x}
\rt)^q.
\end{align*}
Since $v_{I_d}(\ol{x})$ satisfies 
the testing condition \eqref{2.6} 
($(d-1)$-dimensional case), 
by our induction assumption, 
we have that 
\begin{align*}
&\le C c_2^{d-1}
\sum_{I_d\in\cdq(\R)}
\lt(
\int_{\R^{d-1}}
\lt(
\int_{I_d}
f(\ol{x},x_d)\sg(\ol{x},x_d)
\,{\rm d}x_d
v_{I_d}(\ol{x})^{-1}
\rt)^p
v_{I_d}(\ol{x})
\,{\rm d}\ol{x}
\rt)^{\frac{q}{p}}
\\ &= C c_2^{d-1}
\lt[\lt\{
\cdots\cdots
\rt\}^{\frac{p}{q}}
\rt]^{\frac{q}{p}}.
\end{align*}
By integral version of Minkowski's inequality, 
\begin{align*}
&\lt\{
\sum_{I_d\in\cdq(\R)}
\lt(
\int_{\R^{d-1}}
\lt(
\int_{I_d}
f(\ol{x},x_d)\sg(\ol{x},x_d)
\,{\rm d}x_d
v_{I_d}(\ol{x})^{-1}
\rt)^p
v_{I_d}(\ol{x})
\,{\rm d}\ol{x}
\rt)^{\frac{q}{p}}
\rt\}^{\frac{p}{q}}
\\ &\le
\int_{\R^{d-1}}
\lt\{
\sum_{I_d\in\cdq(\R)}
\lt(
\int_{I_d}
f(\ol{x},x_d)\sg(\ol{x},x_d)
\,{\rm d}x_d
v_{I_d}(\ol{x})^{-1}
\rt)^q
v_{I_d}(\ol{x})^{\frac{q}{p}}
\rt\}^{\frac{p}{q}}
\,{\rm d}\ol{x}
\\ &=
\int_{\R^{d-1}}
\lt\{
\sum_{I_d\in\cdq(\R)}
v_{\ol{x}}(I_d)^{\frac{q}{p}}
\lt(
\frac{1}{v_{\ol{x}}(I_d)}
\int_{I_d}
f(\ol{x},x_d)v_{\ol{x}}(x_d)
\,{\rm d}x_d
\rt)^q
\rt\}^{\frac{p}{q}}
\,{\rm d}\ol{x}.
\end{align*}
Since $v_{\ol{x}}(x_d)$ satisfies 
\eqref{2.2} (one-dimensional case), 
by Lemma \ref{lem2.1} 
\begin{align*}
&\le c_2^{\frac{p}{q}}
\int_{\R^{d-1}}\int_{\R}
f(\ol{x},x_d)^p\sg(\ol{x},x_d)
\,{\rm d}x_d\,{\rm d}\ol{x}
\\ &=c_2^{\frac{p}{q}}
\int_{\R^d}f^p\,{\rm d}\sg.
\end{align*}
Altogether, we obtain 
\[
\text{(i)}
\le C c_2^d
\lt(\int_{\R^d}f^p\,{\rm d}\sg\rt)^{\frac{q}{p}}.
\]
This proves the lemma. 
\end{proof}

\section{Proof of Theorem \ref{thm1.1}}\label{sec3}
In what follows we shall prove Theorem \ref{thm1.1}. 

We first notice that, 
if $\sg$ is a~reverse doubling weight on $\R^d$ with $\bt>1$, 
then it satisfies the testing condition \eqref{2.6}. 
Indeed, 
for the dyadic rectangles $R\in\cdr(\R^d)$ 
and $j=1,\ldots,d$,
we have that 
\begin{align*}
\sum_{\substack{
I\in\cdq(\R) \\ I\subset P_j(R)
}}
\sg([R;\,I,j])^{\frac{q}{p}}
&=
\sum_{k=0}^{\8}
\sum_{\substack{
I\in\cdq(\R) \\ I\subset P_j(R),
\,|I|=2^{-k}|P_j(R)|
}}
\sg([R;\,I,j])^{\frac{q}{p}-1}
\sg([R;\,I,j])
\\ &\le
\sum_{k=0}^{\8}
\lt(\frac{1}{\bt^k}\rt)^{\frac{q}{p}-1}
\sg(R)^{\frac{q}{p}-1}
\sum_{\substack{
I\in\cdq(\R) \\ I\subset P_j(R),
\,|I|=2^{-k}|P_j(R)|
}}
\sg([R;\,I,j])
\\ &=
\sg(R)^{\frac{q}{p}}
\sum_{k=0}^{\8}
\lt(\frac{1}{\bt^k}\rt)^{\frac{q}{p}-1}
\\ &= C
\sg(R)^{\frac{q}{p}}.
\end{align*}

The necessity \eqref{1.3} follows at once 
if we substitute the test functions 
$f_i=1_{R}$, $i=1,\ldots,n$, 
into inequality \eqref{1.2}. 
To show that inequality \eqref{1.3} is sufficient, 
we take $q_i>p_i$, $i=1,\ldots,n$, 
with $\sum_{i=1}^n\frac{1}{q_i}=1$. 
This is possible because 
$\sum_{i=1}^n\frac{1}{p_i}>1$.
It follows from testing condition \eqref{1.3} and 
H\"{o}lder's inequality that 
\begin{align*}
&\sum_{R\in\cdr(\R^d)}
K(R)\prod_{i=1}^n\lt|\int_{R}f_i\,{\rm d}\sg_i\rt|
\\ &\le c_2
\sum_{R\in\cdr(\R^d)}
\prod_{i=1}^n
\sg_i(R)^{\frac{1}{p_i}}
\lt(\frac{1}{\sg_i(R)}\int_{R}|f_i|\,{\rm d}\sg_i\rt)
\\ &\le c_2
\prod_{i=1}^n
\lt(
\sum_{R\in\cdr(\R^d)}
\sg_i(R)^{\frac{q_i}{p_i}}
\lt(\frac{1}{\sg_i(R)}\int_{R}|f_i|\,{\rm d}\sg_i\rt)^{q_i}
\rt)^{\frac{1}{q_i}}
\\ &\le C c_2
\prod_{i=1}^n
\|f_i\|_{L^{p_i}(\sg_i)},
\end{align*}
where we have used Lemma \ref{lem2.2} 
by noticing every $\sg_i$ satisfies 
the testing condition \eqref{2.6}.
This completes the proof. 

\section{Proof of Corollary \ref{cor1.2}}\label{sec4}
In what follows we shall prove Corollary \ref{cor1.2}. 

The necessity \eqref{1.5} follows at once 
if we substitute the test functions 
$f_i=1_{R}\sg_i$, $i=1,\ldots,n$, 
into inequality \eqref{1.4}. 
To show that inequality \eqref{1.5} is sufficient, 
we notice that the condition 
\[
\sum_{i=1}^n\frac{1}{p_i}>\frac1q
\]
leads the condition 
\[
\frac{1}{q'}+\sum_{i=1}^n\frac{1}{p_i}>1.
\]
By Theorem \ref{thm1.1} 
we have that the inequality 
\begin{equation}\label{4.1}
\sum_{R\in\cdr(\R^d)}
K(R)\int_{R}g\,{\rm d}\om
\prod_{i=1}^n\int_{R}f_i\,{\rm d}\sg_i
\le C
\|g\|_{L^{q'}(\om)}
\prod_{i=1}^n
\|f_i\|_{L^{p_i}(\sg_i)}
\end{equation}
holds for all nonnegative functions 
$g\in L^{q'}(\om)$ 
and 
$f_i\in L^{p_i}(\sg_i)$, 
provided that the testing condition 
\begin{equation}\label{4.2}
K(R)\om(R)
\prod_{i=1}^n\sg_i(R)
\le C
\om(R)^{\frac{1}{q'}}
\prod_{i=1}^n\sg_i(R)^{\frac{1}{p_i}}
\end{equation}
holds for all dyadic rectangles $R\in\cdr(\R^d)$.

Since \eqref{4.2} is equivalent to our assumption \eqref{1.5}, 
the inequality \eqref{4.1} is proper. 
Rewrite 
$f_i\sg_i=h_i$ 
in \eqref{4.1}, then 
\[
\sum_{R\in\cdr(\R^d)}
K(R)\int_{R}g\,{\rm d}\om
\prod_{i=1}^n\int_{R}h_i\,{\rm d}x
\le C c_2
\|g\|_{L^{q'}(\om)}
\prod_{i=1}^n
\|h_i\|_{L^{p_i}(\sg_i^{1-p_i})}.
\]
This means that 
\[
\int_{\R^d}
gT_{K}(h_1,\ldots,h_n)
\,{\rm d}\om
\le C c_2
\|g\|_{L^{q'}(\om)}
\prod_{i=1}^n
\|h_i\|_{L^{p_i}(\sg_i^{1-p_i})}
\]
and, by duality, 
\[
\|T_{K}(h_1,\ldots,h_n)\|_{L^q(\om)}
\le C c_2
\prod_{i=1}^n
\|h_i\|_{L^{p_i}(\sg_i^{1-p_i})},
\]
which yields the proof. 

\section{Remarks}\label{sec5}
In what follows we give some remarks for strong fractional integral operator. 

For a~number $c>0$ and a~rectangle $R\in\cR$, 
we will use $cR$ to denote the rectangle 
with the same center as $R$ but 
with $c$ times the side-lengths of $R$. 
Let $f_i$, $i=1,\ldots,n$, be locally integrable functions on $R^d$. 
The multilinear strong fractional integral operator 
$I_{\al}(f_1,\ldots,f_n)(x)$, 
$0<\al<dn$ and $x\in\R^d$, 
is given by 
\[
I_{\al}(f_1,\ldots,f_n)(x)
:=
\int_{y_1,\ldots,y_n\in\R^d}
\frac
{f_1(y_1)\cdots f_n(y_n)\,{\rm d}y_1\cdots{\rm d}y_n}
{\lt(\prod_{j=1}^d\max_{i=1}^n|P_j(x)-P_j(y_i)|\rt)^{n-\frac{\al}{d}}},
\]
where $P_j(x)$, $j=1,\ldots,d$, 
is the projection on the $x_j$-axis 
of the point $x\in\R^d$. 

We observe that, 
for $s,t\in\R$ with $s\neq t$, 
the minimal dyadic interval $I\in\cdq(\R)$ 
such that $I\ni s$ and $3I\ni t$ 
satisfies 
\[
\frac{|I|}{2}<|s-t|<2|I|.
\]
This observation and a~calculus of geometric series 
enable us that, 
for any $y_1,\ldots,y_n\neq x$, 
\[
\sum_{R\in\cdr(\R^d)}
|R|^{\frac{\al}{d}-n}
1_{R}(x)
\prod_{i=1}^n1_{3R}(y_i)
\approx
\lt(\prod_{j=1}^d\max_{i=1}^n|P_j(x)-P_j(y_i)|\rt)^{\frac{\al}{d}-n}.
\]
This equation and Fubini's theorem yield 
the precise point-wise relation 
\begin{equation}\label{5.1}
I_{\al}(f_1,\ldots,f_n)(x)
\approx
\sum_{R\in\cdr(\R^d)}
|R|^{\frac{\al}{d}-n}
1_{R}(x)
\prod_{i=1}^n\int_{3R}f_i(y_i)\,{\rm d}y_i,
\quad x\in\R^d.
\end{equation}
Since the right-hand of \eqref{5.1} can be controlled by 
the estimate based upon the finite number of the systems of dyadic rectangles
(see, for example, \cite{LPW}), 
by Corollary \ref{cor1.2}, 
we have the following. 

\begin{prop}\label{prp5.1}
Let $1<p_i<\8$, $1<q<\8$ and 
$\sum_{i=1}^n\frac{1}{p_i}>\frac1q$. 
Let $0<\al<dn$ and 
let $\sg_i$, $i=1,\ldots,n$, 
and $\om$ be reverse doubling weights on $\R^d$. 
The following statements are equivalent\text{:}

\begin{itemize}
\item[{\rm(a)}] 
The weighted norm inequality for multilinear strong fractional integral operator 
\[
\|I_{\al}(f_1,\ldots,f_n)\|_{L^q(\om)}
\le c_1
\prod_{i=1}^n
\|f_i\|_{L^{p_i}(\sg_i^{1-p_i})}
\]
holds for all 
$f_i\in L^{p_i}(\sg_i^{1-p_i})$, 
$i=1,\ldots,n$;
\item[{\rm(b)}] 
The testing condition 
\[
|R|^{\frac{\al}{d}-n}
\om(R)^{\frac1q}
\prod_{i=1}^n\sg_i(R)
\le c_2
\prod_{i=1}^n\sg_i(R)^{\frac{1}{p_i}}
\]
holds for all rectangles $R\in\cR(\R^d)$.
\end{itemize}

\noindent
Moreover,
the least possible constants $c_1$ and $c_2$ are equivalent.
\end{prop}

Letting 
$\om\equiv\sg_1\equiv\cdots\equiv\sg_n\equiv1$, 
we have the following 
Hardy-Littlewood-Sobolev inequality 
for strong fractional integral operator. 

\begin{prop}\label{prp5.2}
Let $1<q<\8$, $1<p_i<\8$, 
$0<\al<dn$ and 
\[
\frac1q=\sum_{i=1}^n\frac{1}{p_i}-\frac{\al}{d}.
\]
Then the multilinear norm inequality 
\[
\|I_{\al}(f_1,\ldots,f_n)\|_{L^q(\R^d)}
\le C
\prod_{i=1}^n
\|f_i\|_{L^{p_i}(\R^d)}
\]
holds for all 
$f_i\in L^{p_i}(\R^d)$, 
$i=1,\ldots,n$.
\end{prop}

\end{document}